\documentclass[12pt, reqno]{amsart}
\usepackage{amssymb,latexsym,amsmath,amsfonts}
\usepackage{latexsym}
\usepackage[mathscr]{eucal}

8.7in \numberwithin{equation}{section}

\def ~{\hspace{1mm}}

\newtheorem{thm}{Theorem}[section]

\newtheorem{lem}[thm]{Lemma}

\newtheorem{prop}[thm]{Proposition}
\newtheorem{defn}[thm]{Definition}
\newtheorem{rem}[thm]{Remark}

\numberwithin{equation}{section}

\def\textmatrix#1&#2\\#3&#4\\{\bigl({#1 \atop #3}\ {#2 \atop #4}\bigr)}
\def\dispmatrix#1&#2\\#3&#4\\{\left({#1 \atop #3}\ {#2 \atop #4}\right)}

\begin{document}

\title[Functional model]
{A functional model for pure $\Gamma$-contractions}

\author[Bhattacharyya]{Tirthankar Bhattacharyya}
\address[Bhattacharyya]{Department of Mathematics, Indian Institute of Science, Banaglore 560 012}
\email{tirtha@member.ams.org}

\author[Pal]{Sourav Pal}
\address[Pal]{Department of Mathematics, Indian Institute of Science, Banaglore 560 012}
\email{souravmaths@gmail.com}

\maketitle

\begin{abstract}
A pair of commuting operators $(S,P)$ defined on a Hilbert space
$\mathcal H$ for which the closed symmetrized bidisc $$ \Gamma= \{
(z_1+z_2,z_1z_2): |z_1|\leq 1,\, |z_2|\leq 1 \}\subseteq \mathbb
C^2,
$$ is a spectral set is called a $\Gamma$-contraction in the
literature. A $\Gamma$-contraction $(S,P)$ is said to be pure if
$P$ is a pure contraction, i.e, ${P^*}^n \rightarrow 0$ strongly
as $n \rightarrow \infty $. Here we construct a functional model
and produce a set of unitary invariants for a pure
$\Gamma$-contraction. The key ingredient in these constructions is
an operator, which is the unique solution of the operator equation
$$ S-S^*P=D_PXD_P, \textup{ where } X\in \mathcal B(\mathcal D_P),
$$and is called the fundamental operator of the
$\Gamma$-contraction $(S,P)$. We also discuss some important
properties of the fundamental operator.
\end{abstract}

\section{INTRODUCTION AND PRELIMINARIES}
The closed symmetrized bidisc $\Gamma$ is polynomially convex.
Thus, a pair of commuting bounded operators $(S,P)$ is a
$\Gamma$-contraction if and only if $\|p(S,P)\|\leq \|p\|_{\infty,
\Gamma},$ for any polynomial $p$. The $\Gamma$-contractions were
introduced by Agler and Young in \cite{ay-jfa} and have been
thoroughly studied in \cite{ay-ems}, \cite{ay-jot} and
\cite{tirtha-sourav}. An understanding of this family of operator
pairs has led to the solution of a special case of the spectral
Nevanlina-Pick problem(\cite{ay-ieot},\cite{ay-tams}), which is
one of the problems that arise in $H^{\infty}$ control
theory(\cite{BAF}). Also they play a pivotal role in the study of
complex geometry of the set $\Gamma$(see
\cite{ay-blm},\cite{ay-jga}).\\
Spectral sets and complete spectral sets for a bounded operator
$T$ on a Hilbert space $\mathcal H$ or for a tuple of bounded
operators have been well-studied for long and several important
results are known(see
\cite{tirtha-gadadhar},\cite{DMS},\cite{paulsen}). Dilation theory
for an operator or a tuple of operators is well-studied too and
has made some rapid progress in the last twenty years through
Arveson (\cite{arveson}), Popescu
(\cite{popescu1},\cite{popescu2}), Muller and Vasilescu
(\cite{muller}), Pott (\cite{pott}) and others.

Sz.-Nagy and Foias developed the model theory for a
contraction(\cite{nazy}). They found the minimal unitary dilation
of a contraction and it has become a powerful tool for studying an
arbitrary contraction. By von Neumann's inequality, an operator
$T$ is a contraction if and only if $\|p(T)\|\leq \|p\|_{\infty,
\mathbb D}$ for all polynomials $p$, $\mathbb D$ being the open
unit disc in the complex plane. This property itself is very
beautiful and so is the concept of spectral set of an operator. A
compact subset $X$ of $\mathbb C$ is called a spectral set for an
operator $T$ if
$$ \|\pi(T)\|\leq \displaystyle \sup_{z\in X} \|\pi(z)\|=\|\pi\|_{\infty, X}\,, $$
for all rational functions $\pi$ with poles off $X$. If the above
inequality holds for matrix valued rational functions $\pi$, then
$X$ is called a complete spectral set for the operator $T$.
Moreover, $T$ is said to have a normal $\partial X$-dilation if
there is a Hilbert space $\mathcal{K}$ containing $\mathcal{H}$ as
a subspace and a normal operator $N$ on $\mathcal{K}$ with
$\sigma(N)\subseteq
\partial X$ such that $$\pi(T)=P_{\mathcal H}\pi(N)|_{\mathcal H},$$
for all rational functions $\pi$ with poles off $X$. It is a
remarkable consequence of Arveson's extension theorem that $X$ is
a complete spectral set for $T$ if and only if $T$ has a normal
$\partial X$-dilation. Rephrased in this language, the Sz.-Nagy
dilation theorem says that if $\mathbb{D}$ is a spectral set for
$T$ then $T$ has a normal $\partial \mathbb{D}$-dilation. For $T$
to have a normal $\partial X$-dilation it is necessary that  $X$
be a spectral set for $T$. Sufficiency has been investigated for
many domains in $\mathbb{C}$ and several interesting results are
known including success of such a dilation on an
annulus(\cite{agler-ann}) and its failure in triply connected
domains(\cite{ahr},\cite{DM}). When $(T_1,T_2)$ is a commuting
pair of operators for which $\mathbb{D}^2$ is a spectral set,
Ando's theorem provides a simultaneous commuting unitary dilation
of $(T_1,T_2)$. Such classically beautiful concepts led Agler and
Young to the following definitions.

\begin{defn} A commuting pair $(S,P)$ is called a $\Gamma$-{\em unitary}
if $S$ and $P$ are normal operators and the joint spectrum
$\sigma(S,P)$ of $(S,P)$ is contained in the distinguished
boundary $b\Gamma$ defined by
$$ b\Gamma=\{ (z_1+z_2,z_1z_2):\; |z_1|=|z_2|=1 \}\subseteq
\Gamma.$$
\end{defn}
\begin{defn}
A commuting pair $(\tilde{S},\tilde{P})$ on $\mathcal N$ is said
to be a $\Gamma$-{\em unitary extension} of a $\Gamma$-contraction
$(S,P)$ on $\mathcal H$ if $\mathcal H\subseteq \mathcal N$,
$(\tilde{S},\tilde{P})$ is a $\Gamma$-unitary, $\mathcal H$ is a
common invariant subspace of both $\tilde{S}$ and $\tilde{P}$ and
$\tilde{S}|_{\mathcal H}=S, \tilde{P}|_{\mathcal H}=P$.
\end{defn}

\begin{defn}
A commuting pair $(S,P)$ is called a $\Gamma$-{\em isometry} if it
has a $\Gamma$-unitary extension. A commuting pair $(S,P)$ is a
$\Gamma$-{\em co-isometry} if $(S^*,P^*)$ is a $\Gamma$-isometry.
\end{defn}
\begin{defn}
Let $(S,P)$ be a $\Gamma$-contraction on $\mathcal H$. A pair of
commuting operators $(T,V)$ acting on a Hilbert space $\mathcal N
\supseteq \mathcal H$ is called a $\Gamma$-{\em isometric
dilation} of $(S,P)$ if $(T,V)$ is a $\Gamma$-isometry, $\mathcal
H$ is a co-invariant subspace of both $T$ and $V$ and
$T^*|_{\mathcal H}=S^*, V^*|_{\mathcal H}=P^*$. Moreover, the
dilation will be called minimal if $$\mathcal
N=\overline{\textup{span}} \{V^nh:h\in\mathcal H \mbox{ and
}n=0,1,2,\dots\}.$$ Thus $(T,V)$ is a $\Gamma$-isometric dilation
of a $\Gamma$-contraction $(S,P)$ if and only if $(T^*,V^*)$ is a
$\Gamma$-co-isometric extension of $(S^*,P^*)$.
\end{defn}
A $\Gamma$-contraction $(S,P)$ acting on a Hilbert space $\mathcal
H$ is said to be pure if $P$ is a pure contraction, i.e,
${P^*}^n\rightarrow 0$ strongly as $n\rightarrow \infty$. The
class of pure $\Gamma$-contractions plays a pivotal role in
deciphering the structure of a class of $\Gamma$-contractions. In
\cite{ay-jot} (Theorem 2.8), Agler and Young proved that every
$\Gamma$-contraction $(S,P)$ acting on a Hilbert space $\mathcal
H$ can be decomposed into two parts $(S_1,P_1)$ and $(S_2,P_2)$ of
which $(S_1,P_1)$ is a $\Gamma$-unitary and $(S_2,P_2)$ is a
$\Gamma$-contraction with $P$ being a completely non-unitary
contraction. This shows an analogy with the decomposition of a
single contraction. Indeed, if $\mathcal H_1$ is the maximal
subspace of $\mathcal H$ which reduces $P$ and on which $P$ is
unitary, then $\mathcal H_1$ reduces $S$ as well and $(S_1,P_1)$
is same as $(S|_{\mathcal H_1},P|_{\mathcal H_1})$. Also both $S$
and $P$ are reduced by the subspace $\mathcal H_2$, the
orthocomplement of $\mathcal H_1$ in $\mathcal H$, and $(S_2,P_2)$
is same as $(S|_{\mathcal H_2},P|_{\mathcal H_2})$. The functional
model and unitary invariants we produce here give a good vision of
those $\Gamma$-contractions $(S,P)$ for which the part $(S_2,P_2)$
described above is a pure $\Gamma$-contraction.

The program that Sz.-Nagy and Foias carried out for a contraction
had two parts. The dilation was the first part which was followed
by a functional model and a complete unitary invariant. For a
$\Gamma$-contraction, the first part of that program was carried
out in \cite{ay-jot} by Agler and Young. The second half is the
content of this article.

For a contraction $P$ defined on a Hilbert space $\mathcal H$, let
$\Lambda_P$ be the set of all complex numbers for which the
operator $I-zP^*$ is invertible. For $z\in \Lambda_P$, the
characteristic function of $P$ is defined as
\begin{eqnarray}\label{e0} \Theta_P(z)=[-P+zD_{P^*}(I-zP^*)^{-1}D_P]|_{\mathcal D_P}.
\end{eqnarray} Here the operators $D_P$ and $D_{P^*}$ are the
defect operators $(I-P^*P)^{1/2}$ and $(I-PP^*)^{1/2}$
respectively. By virtue of the relation $PD_P=D_{P^*}P$ (section
I.3 of \cite{nazy}), $\Theta_P(z)$ maps $\mathcal
D_P=\overline{\textup{Ran}}D_P$ into $\mathcal
D_{P^*}=\overline{\textup{Ran}}D_{P^*}$ for every $z$ in
$\Lambda_P$.

For a pair of commuting bounded operators $S,P$ on a Hilbert space
$\mathcal H$ with $\|P\|\leq 1$, we introduced in
\cite{tirtha-sourav}, the notion of the fundamental equation. For
the pair $S,P$ it is defined as \begin{eqnarray}\label{e1}
S-S^*P=D_PXD_P,\quad X\in \mathcal B(\mathcal D_P),
\end{eqnarray} and the same for the pair $S^*,P^*$ is
\begin{eqnarray}\label{e2} S^*-SP^*=D_{P^*}YD_{P^*},\quad Y\in \mathcal B(\mathcal D_{P^*}).\end{eqnarray}
In the same paper we also proved the existence and uniqueness of
solutions of such equations when $(S,P)$ is a $\Gamma$-contraction
(Theorem 4.2 of \cite{tirtha-sourav}). The unique solution was
named the fundamental operator of the $\Gamma$-contraction because
it led us to a new characterization for $\Gamma$-contractions
(Theorem 4.4 of \cite{tirtha-sourav}).

In section 2, we discuss some interesting properties of the
fundamental operator. In section 3, we construct a functional
model for a pure $\Gamma$-contraction $(S,P)$ and this is the main
content of this paper. The fundamental operator $F_*$ of
$(S^*,P^*)$ is taken as the key ingredient in that construction.
In section 4, we produce a set of unitary invariants for pure
$\Gamma$-contractions. For the unitary equivalence of two pure
$\Gamma$-contractions $(S,P)$ and $(S_1,P_1)$ on Hilbert spaces
$\mathcal H$ and $\mathcal H_1$ respectively, we produce here a
set of unitary invariants which consists of two things mainly. The
first one demands the coincidence of the characteristic functions
of $P$ and $P_1$. The second condition is the unitary equivalence
of the fundamental operators ${F_*}$ and ${F_*}_1$ of $(S^*,P^*)$
and $(S_1^*,P_1^*)$ by the same unitary from $\mathcal D_{P^*}$ to
$\mathcal D_{P_1^*}$ that is involved in establishing the
coincidence of the characteristic functions of $P$ and $P_1$.

\section{AUTOMORPHISMS AND THE FUNDAMENTAL OPERATOR}

For a $\Gamma$-contraction $(S,P)$ we find out an explicit form of
the fundamental operator of $\tau (S,P)$, where $\tau$ is an
automorphism of the open symmetrized bidisc
$$ G=\{ (z_1+z_2,z_1z_2):\; | z_1|<1, |z_2| <1 \}. $$ It is
well-known, see \cite{ay-caot} and \cite{jarnicki}, that any
automorphism $\tau$ of $G$ is given as
follows:\begin{eqnarray}\label{f1}
\tau(z_1+z_2,z_1z_2)=\tau_m(z_1+z_2,z_1z_2)=(m(z_1)+m(z_2),m(z_1)m(z_2)),\;
z_1,z_2 \in\mathbb D,\end{eqnarray} where $m$ is an automorphism
of the disc $\mathbb D$. Recall that the joint spectrum
$\sigma(S,P)$ of a $\Gamma$-contraction $(S,P)$ is contained in
$\Gamma$. Thus if $\tau$ is a $\mathbb C^2$-valued holomorphic map
in a neighbourhood $\mathbf N(\Gamma)$ of $\Gamma$ mapping
$\Gamma$ into itself, then by functional calculus (see
\cite{taylor3}), $(S_{\tau},P_{\tau}):=\tau(S,P)$ is well defined
as a pair of commuting bounded operators.
\begin{lem}
For $(S,P)$ and $\tau$ as above, $(S_{\tau},P_{\tau})$ is a
$\Gamma$-contraction.
\end{lem}
\begin{proof}
To show that $\Gamma$ is a spectral set of $(S_{\tau},P_{\tau})$.
Let $f$ be a polynomial over $\mathbb C$ in two variables. Then
$$\|f(S_{\tau},P_{\tau})\|=\|f\circ \tau(S,P)\|\leq\|f\circ \tau \|_{\infty,\Gamma}=\sup_{z\in\Gamma}|f(\tau(z))|\leq
\|f\|_{\infty,\Gamma}\,,$$ since $\tau(z)\in\Gamma$ for all
$z\in\Gamma$ and hence $(S_{\tau},P_{\tau})$ is a
$\Gamma$-contraction.
\end{proof}
The following is the main result of this section.
\begin{thm}
Let $(S,P)$ be a $\Gamma$-contraction defined on a Hilbert space
$\mathcal{H}$ and let $\tau$ be an automorphism of $G$. Let
$\tau=\tau_m$ as in (\ref{f1}) and $m$ be given by $m(z)=\beta
\frac{z-a}{1-\bar{a}z}$ for some $a\in\mathbb D$ and
$\beta\in\mathbb T$. Let $F$ and $F_{\tau}$ be the fundamental
operators of $(S,P)$ and $(S_{\tau},P_{\tau})$ respectively. Then
there is a unitary $U:\mathcal{D}_{P_{\tau}}\rightarrow
\mathcal{D}_P$ such that
$$F_{\tau}=U^*((1+|a|^2)-\bar{a}F-aF^*)^{-\frac{1}{2}}\beta
(F+a^2F^*-2a)((1+|a|^2)-\bar{a}F-aF^*)^{-\frac{1}{2}}U.$$
\end{thm}
\begin{proof}
We have
\begin{align*}
\tau(s,p)=\tau(z_1+z_2,z_1z_2)&=(\beta(\frac{z_1-a}{1-\bar{a}z_1}+\frac{z_2-a}{1-\bar{a}z_2}),\beta^2
\frac{(z_1-a)(z_2-a)}{(1-\bar{a}z_1)(1-\bar{a}z_2)})\\&=(\beta
\frac{(z_1+z_2)-2\bar{a}z_1z_2+|a|^2(z_1+z_2)-2a}{1-\bar{a}(z_1+z_2)+\bar{a}^2z_1z_2},
\beta^2
\frac{z_1z_2-a(z_1+z_2)+a^2}{1-\bar{a}(z_1+z_2)+\bar{a}^2z_1z_2})\\&=(\beta\frac{(1+|a|^2)s-2\bar{a}p-2a}{1-\bar{a}s+\bar{a}^2p},
\beta^2\frac{p-as+a^2}{1-\bar{a}s+\bar{a}^2p}).
\end{align*}
It is obvious that $\tau$ can be defined on the open set
$\Gamma_a=\{ (z_1+z_2,z_1z_2):\;
|z_1|<\frac{1}{|a|}\,,|z_2|<\frac{1}{|a|} \}$, which contains
$\Gamma$. Clearly
$$(S_{\tau},P_{\tau})=\tau(S,P)=(\beta((1+|a|^2)S-2\bar{a}P-2a)(I-\bar{a}S+\bar{a}^2P)^{-1},\beta^2(P-aS+a^2)(I-\bar{a}S+\bar{a}^2P)^{-1}).$$
Here
\begin{align*}
{D}^2_{P_{\tau}}=&(I-P_{\tau}^*P_{\tau})\\
=&I-(I-aS^*+a^2P^*)^{-1}(P^*-\bar{a}S^*+\bar{a}^2)(P-aS+a^2)(I-\bar{a}S+\bar{a}^2P)^{-1}\\
=&(I-aS^*+a^2P^*)^{-1}[(I-aS^*+a^2P^*)(I-\bar{a}S+\bar{a}^2P)-\\&(P^*-\bar{a}S^*+\bar{a}^2)(P-aS+a^2)](I-\bar{a}S+\bar{a}^2P)^{-1}\\
=&(I-aS^*+a^2P^*)^{-1}[-\bar{a}(1-|a|^2)(S-S^*P)-a(1-|a|^2)(S^*-P^*S)\\&
(1-|a|^4)(I-P^*P)](I-\bar{a}S+\bar{a}^2P)^{-1}\\
=&(1-|a|^2)(I-aS^*+a^2P^*)^{-1}[(1+|a|^2)(I-P^*P)-\\& \bar{a}(S-S^*P)-a(S^*-P^*S)](I-\bar{a}S+\bar{a}^2P)^{-1}\\
=&(1-|a|^2)(I-aS^*+a^2P^*)^{-1}[(1+|a|^2){D}_P^2-\bar{a}{D}_PF{D}_P-a{D}_PF^*{D}_P](I-\bar{a}S+\bar{a}^2P)^{-1},\\&
\mbox{since } S-S^*P={D}_PF{D}_P\\
=&(1-|a|^2)(I-aS^*+a^2P^*)^{-1}{D}_P[(1+|a|^2)-\bar{a}F-aF^*]{D}_P(I-\bar{a}S+\bar{a}^2P)^{-1}
\end{align*}
Now we show that the operator $(1+|a|^2)-\bar{a}F-aF^*$ defined on
$\mathcal{D}_P$ is invertible. Since
$F\in\mathcal{B}(\mathcal{D}_P)$, it is enough to show that
$(1+|a|^2)-\bar{a}F-aF^*$ is bounded below, i.e,
$$\inf_{\|x\|\leq1}\langle ((1+|a|^2)-\bar{a}F-aF^*)x,x \rangle>0,$$
or equivalently $$\sup_{\|x\|\leq1}|\bar{a}\langle Fx,x
\rangle+a\langle F^*x,x \rangle|<(1+|a|^2).$$ Since the numerical
radius of $F$ is not greater than $1$,
$$\sup_{\|x\|\leq1}|\bar{a}\langle Fx,x \rangle+a\langle F^*x,x
\rangle|\leq2|a|<(1+|a|^2)$$ as $1+|a|^2-2|a|=(1-|a|)^2>0$ for
$a\in\mathbb{D}$ and consequently the operator
$(1+|a|^2-\bar{a}F-aF^*)$ is invertible.\\
Let
$X=(1-|a|^2)^{\frac{1}{2}}[(1+|a|^2)-\bar{a}F-aF^*]^{\frac{1}{2}}{D}_P(I-\bar{a}S+\bar{a}^2P^*)^{-1}.$
Then $X$ is an operator from $\mathcal H$ to $\mathcal {D}_P$.
Also ${D}_{P_{\tau}}^2=X^*X$ and
$\overline{\textup{Ran}}X=\mathcal{D}_P$ as
$(1+|a|^2)-\bar{a}F-aF^*$ is invertible. Now define
\begin{align*}
U:&\mathcal{D}_{P_{\tau}}\rightarrow
\overline{\textup{Ran}}X=\mathcal{D}_P
\\& {D}_{P_{\tau}}h\mapsto Xh.
\end{align*}
Clearly $U$ is onto. Moreover,
$$\|U{D}_{P_{\tau}}h\|^2=\|Xh\|^2=\langle X^*Xh,h
\rangle=\langle {D}_{P_{\tau}}^2h,h
\rangle=\|{D}_{P_{\tau}}h\|^2.$$ So $U$ is a surjective isometry
i.e, a unitary. Also {\small\begin{align*}
S_{\tau}-S^*_{\tau}P_{\tau}=&\beta[((1+|a|^2)S-2\bar{a}P-2a)(I-\bar{a}S+\bar{a}^2P)^{-1}-\\&
(I-aS^*+a^2P^*)^{-1}((1+|a|^2)S^*-2aP^*-2\bar{a})
(P-aS+a^2)(I-\bar{a}S+\bar{a}^2P)^{-1})]\\=&
(I-aS^*+a^2P^*)^{-1}\beta[(I-aS^*+a^2P^*)((1+|a|^2)S-2\bar{a}P-2a)-
\\&((1+|a|^2)S^*-2aP^*-2\bar{a})(P-aS+a^2)](I-\bar{a}S+\bar{a}^2P)^{-1}\\=&
(I-aS^*+a^2P^*)^{-1}\beta[(1-|a|^2)(S-S^*P)+2a^2(S^*-P^*S)-a^2(1+|a|^2)(S^*-P^*S)\\&
-2a(I-P^*P)+2a|a|^2(I-P^*P)](I-\bar{a}S+\bar{a}^2P)^{-1}\\=&
(I-aS^*+a^2P^*)^{-1}\beta[(1-|a|^2)(S-S^*P)+a^2(1-|a|^2)(S^*-P^*S)-\\&
2a(1-|a|^2)(I-P^*P)](I-\bar{a}S+\bar{a}^2P)^{-1}\\=&
(1-|a|^2)(I-aS^*+a^2P^*)^{-1}\beta[(S-S^*P)+a^2(S^*-P^*S)-2a(I-P^*P)](I-\bar{a}S+\bar{a}^2P)^{-1}\\=&
(1-|a|^2)(I-aS^*+a^2P^*)^{-1}\beta[{D}_PF{D}_P+a^2{D}_PF^*{D}_P-2a{D}_P^2](I-\bar{a}S+\bar{a}^2P)^{-1},\\&
\mbox{since } S-S^*P={D}_PF{D}_P\\=&
(1-|a|^2)(I-aS^*+a^2P^*)^{-1}\beta{D}_P[F+a^2F^*-2a]{D}_P(I-\bar{a}S+\bar{a}^2P)^{-1}\\=&
X^*[((1+|a|^2)-\bar{a}F-aF^*)^{-\frac{1}{2}}\beta(F+a^2F^*-2a)((1+|a|^2)-\bar{a}F-aF^*)^{-\frac{1}{2}}]X\\=&
{D}_{P_{\tau}}U^*[((1+|a|^2)-\bar{a}F-aF^*)^{-\frac{1}{2}}\beta(F+a^2F^*-2a)((1+|a|^2)-\bar{a}F-aF^*)^{-\frac{1}{2}}]U{D}_{P_{\tau}}.
\end{align*}}
Again since
$S_{\tau}-S_{\tau}^*P_{\tau}={D}_{P_{\tau}}F_{\tau}{D}_{P_{\tau}}$
and $F_{\tau}$ is unique, we have
$$F_{\tau}=U^*((1+|a|^2)-\bar{a}F-aF^*)^{-\frac{1}{2}}\beta
(F+a^2F^*-2a)((1+|a|^2)-\bar{a}F-aF^*)^{-\frac{1}{2}}U.$$
\end{proof}

Here is an interesting result which relates the fundamental
operator of a $\Gamma$-contraction $(S,P)$ with that of
$(S^*,P^*)$.
\begin{prop}
Let $(S,P)$ be a $\Gamma$-contraction on $\mathcal H$ and let $F,
{F_*}$ be the fundamental operators of $(S,P)$ and $(S^*,P^*)$
respectively. Then $PF={F_*}^*P|_{\mathcal D_P}$.
\end{prop}
\begin{proof}
Since $F\in \mathcal B(\mathcal D_P)$ and $F_*\in \mathcal
B(\mathcal D_{P^*})$, both $PF$ and ${F_*}^*P|_{\mathcal D_P}$ are
in $\mathcal B(\mathcal D_P,\mathcal D_{P^*})$. For $D_Ph \in
\mathcal D_P$ and $D_{P^*}h^{\prime}\in \mathcal D_{P^*}$, we have
\begin{align*}
\langle PFD_Ph,D_{P^*}h^{\prime} \rangle &=\langle
D_{P^*}PFD_Ph,h^{\prime} \rangle \\& =\langle
PD_PFD_Ph,h^{\prime}\rangle, \quad \textup{since }PD_P=D_{P^*}P \\
&=\langle P(S-S^*P)h,h^{\prime} \rangle,\quad \textup{since }
S-S^*P=D_PFD_P \\&=\langle (PS-PS^*P)h,h^{\prime} \rangle
\\&=\langle (SP-PS^*P)h,h^{\prime} \rangle \\&=\langle (S-PS^*)Ph,h^{\prime}
\rangle\\&=\langle D_{P^*}{F_*}^*D_{P^*}Ph,h^{\prime} \rangle,
\quad \textup{since } S^*-SP^*=D_{P^*}{F_*}D_{P^*}
\\&=\langle {F_*}^*PD_Ph,D_{P^*}h^{\prime} \rangle.
\end{align*}
Hence $PF={F_*}^*P|_{\mathcal D_P}$.
\end{proof}

 \section{FUNCTIONAL MODEL}

 In \cite{nazy}, Sz.-Nagy and Foias showed that every
pure contraction $P$ defined on a Hilbert space $\mathcal H$ is
unitarily equivalent to the operator $\mathbb P=P_{\mathbb
H_P}(M_z\otimes I)|_{\mathcal D_{P^*}}$ on the Hilbert space
$\mathbb H_P=(H^2(\mathbb D)\otimes \mathcal D_{P^*}) \ominus
M_{\Theta_P}(H^2(\mathbb D)\otimes \mathcal D_P)$, where $M_z$ is
the multiplication operator on $H^2(\mathbb D)$ and $M_{\Theta_P}$
is the multiplication operator from $H^2(\mathbb D)\otimes
\mathcal D_P$ into $H^2(\mathbb D)\otimes \mathcal D_{P^*}$
corresponding to the multiplier $\Theta_P$, which is the
characteristic function of $P$ defined in section 1. This is known
as Sz.Nagy-Foias model for a pure contraction. Here analogously we
produce a model for a pure $\Gamma$-contraction.
\begin{thm}\label{modelthm}
Every pure $\Gamma$-contraction $(S,P)$ defined on a Hilbert space
$\mathcal H$ is unitarily equivalent to the pair $(S_1,P_1)$ on
the Hilbert space $\mathbb H_P=(H^2(\mathbb D)\otimes \mathcal
D_{P^*})\ominus M_{\Theta_P}(H^2(\mathbb D)\otimes \mathcal D_P)$
defined as $S_1=P_{\mathbb H_P}(I\otimes {F_*}^*+M_z\otimes
{F_*})|_{\mathbb H_P}$ and $P_1=P_{\mathbb H_P}(M_z\otimes
I)|_{\mathbb H_P}.$
\end{thm}
\begin{rem}
\textup{It is interesting to see here that the model space for a
pure $\Gamma$-contraction $(S,P)$ is same as that of $P$ and the
model operator for $P$ is the same given in Sz.-Nagy-Foias model.}
\end{rem}
To prove the above theorem, we define an operator $W$ in the
following way: \begin{align*} W: & \mathcal{H} \rightarrow
H^2(\mathbb D)\otimes \mathcal D_{P^*} \\& h\mapsto
\sum_{n=0}^{\infty} z^n\otimes D_{P^*}{P^*}^n
 h. \end{align*} It is obvious that $W$ embeds $\mathcal H$ isometrically
 inside $H^2(\mathbb D)\otimes \mathcal D_{P^*}$ (see proof of Theorem 4.6 of
 \cite{tirtha-sourav}) and its adjoint $L: H^2(\mathbb
D)\otimes \mathcal D_{P^*} \rightarrow \mathcal{H}$ is given by
 $$ L(f\otimes \xi)=f(P)D_{P^*}\xi, \quad \textup{for all } f \in \mathbb
C[z], \textup{ and } \xi \in \mathcal D_{P^*}. $$ Here we mention
an interesting and well-known property of the operator $L$ which
we use to prove the above theorem.
\begin{lem}\label{easylem}
 For a pure contraction $P$, the identity $$L^*L+M_{\Theta_P}M_{\Theta_P}^*=I_{H^2(\mathbb D)\otimes \mathcal
 D_{P^*}}$$ holds.
\end{lem}
\begin{proof}
As observed by Arveson in the proof of Theorem 1.2 in
 \cite{arveson2}, the operator $L$ satisfies the identity $$L(k_z\otimes \xi)=(I-\bar z P)^{-1}D_{P^*}\xi \quad \mbox{for } z\in \mathbb D, \xi \in \mathcal
 D_{P^*},$$ where $k_z(w)=(1-\langle w,z \rangle)^{-1}$.
 Therefore, for $z,w$ in $\mathbb{D}$ and $\xi,\eta$ in $\mathcal
 D_{P^*}$, we obtain that
 \begin{align*}
 &\quad \langle (L^*L+M_{\Theta_P}M_{\Theta_P}^*)k_z\otimes \xi, k_w\otimes \eta
 \rangle\\&
 =\langle L(k_z\otimes \xi),L(k_w\otimes \eta) \rangle +\langle M_{\Theta_P}^*(k_z\otimes \xi), M_{\Theta_P}^*(k_w\otimes \eta)
 \rangle\\&
 =\langle(I-\bar{z}P)^{-1}D_{P^*}\xi,(I-\bar{w}P)^{-1}D_{P^*}\eta \rangle+\langle k_z\otimes \Theta_P(z)^*\xi,k_w\otimes \Theta_P(w)^*\eta
 \rangle\\&
 =\langle D_{P^*}(I-wP^*)^{-1}(I-\bar{z}P)^{-1}D_P^*\xi,\eta \rangle+\langle k_z,k_w \rangle\langle \Theta_P(w)\Theta_P(z)^*\xi,\eta
 \rangle\\&
 =\langle k_z\otimes \xi, k_w\otimes \eta \rangle.
 \end{align*}
 The last equality follows from the following well-known identity,
 $$ 1-\Theta_P(w)\Theta_P(z)^*=(1-w\bar{z})D_{P^*}(1-wP^*)^{-1}(1-\bar{z}P)^{-1}D_{P^*}, $$ where $\Theta_P$ is the characteristic function of $P$.
 Using the fact that the vectors $k_z$ forms a total
 set in $H^2(\mathbb{D})$, the assertion follows.
\end{proof}
\textit{Proof of Theorem} \ref{modelthm}. It is evident from Lemma
\ref{easylem} that $$ L^*(\mathcal H)=W(\mathcal H)=\mathbb
H_P=(H^2(\mathbb D)\otimes \mathcal D_{P^*})\ominus
M_{\Theta_P}(H^2(\mathbb D)\otimes \mathcal D_P).$$ Let $T=I
\otimes F_*^*+M_z\otimes F_*$ and $V=M_z\otimes I$. For a basis
vector $z^n\otimes \xi$ of $H^2(\mathbb D)\otimes \mathcal
D_{P^*}$ and $h\in \mathcal H$ we have
$$ \langle L(z^n\otimes \xi),h \rangle  = \langle z^n \otimes \xi ,
\displaystyle \sum_{k=0}^{\infty}{z^k \otimes {D}_{P^*}{P^*}^kh}
\rangle  = \langle \xi , {D}_{P^*}{P^*}^nh \rangle  = \langle P^n
{D}_{P^*}\xi , h \rangle. $$ This implies that
$$L(z^n \otimes \xi)=P^n {D}_{P^*} \xi, \quad
\textup{for}\; n=0,1,2,3,...$$ Therefore
$$ \langle L(M_z \otimes I)(z^n \otimes \xi), h \rangle = \langle
z^{n+1} \otimes \xi , \displaystyle \sum_{k=0}^{\infty}{z^k
\otimes {D}_{P^*}{P^*}^kh} \rangle  = \langle \xi,
{D}_{P^*}{P^*}^{n+1}h \rangle = \langle P^{n+1} {D}_{P^*} \xi,h
\rangle. $$ Consequently, $LV = PL$ on vectors of the form $z^n
\otimes \xi$ which span $H^2 \otimes {\mathcal D}_{P^*}$ and hence
$$LV = PL.$$ Therefore $V^*$ leaves the range of $L^*$ (isometric
copy of $\mathcal H$) invariant and
$V^*|_{L^*\mathcal{H}}=L^*P^*L$ which is the copy of the operator
$P^*$ on range of $L^*$. Also
\begin{align*} LT(z^n \otimes \xi)=L(I \otimes F_*^*+M_z \otimes
F_*)(z^n \otimes \xi) &=L(I \otimes F_*^*)(z^n \otimes \xi)+L(M_z
\otimes F_*)(z^n \otimes \xi)\\&=L(z^n \otimes F_*^*\xi)+L(z^{n+1}
\otimes F_* \xi)\\&=P^n {D}_{P^*}F_*^*\xi + P^{n+1}
{D}_{P^*}F_*\xi.
\end{align*} Again $SL(z^n \otimes \xi)=SP^n {D}_{P^*} \xi$.
Therefore for showing $LT=SL$, it is enough to show that
\begin{align*}
& P^n {D}_{P^*}F_*^*+P^{n+1} {D}_{P^*}F_*=SP^n {D}_{P^*}=P^nS {D}_{P^*}\\
& \mbox{i.e,}\; {D}_{P^*}F_*^*+P {D}_{P^*}F_*=S {D}_{P^*}.
\end{align*}
Let $H={D}_{P^*}F_*^*+P {D}_{P^*}F_*-S {D}_{P^*}$. Then $H$ is
defined from $\mathcal D_{P^*}\rightarrow \mathcal H$. Since $F_*$
is a solution of (\ref{e2}), we have \begin{align*} HD_{P^*}
&=D_{P^*}F_*^*D_{P^*}+PD_{P^*}F_*D_{P^*}-SD_{P^*}^2 \\ &
=(S-PS^*)+P(S^*-SP^*)-S(I-PP^*)\\&=0.
\end{align*} Hence $H=0$.
So we have
$${D}_{P^*}F_*^*+P{D}_{P^*}F_*=S{D}_{P^*}$$ and therefore
$$ L(I \otimes F_*^*+M_z \otimes F_*)=SL.$$ This shows
that $T^*$ leaves $L^*(\mathcal H)$ invariant as well as
$T^*|_{L^*(\mathcal H)}=L^*S^*L$. Thus $\mathbb H_P$ is
co-invariant under $I \otimes F_*^*+M_z\otimes F_*$
 and $M_z\otimes I$. Hence $\mathbb H_P$ is a model space and
 $P_{\mathbb H_P}(I\otimes {F_*}^*+M_z\otimes
{F_*})|_{\mathbb H_P}$ and $P_{\mathbb H_P}(M_z\otimes
I)|_{\mathbb H_P}$ are model
 operators for $S$ and $P$ respectively.
\qed

\section{A SET OF UNITARY INVARIANTS FOR PURE $\Gamma$-CONTRACTIONS}
The characteristic function of a contraction is a classical
complete unitary invariant devised by Sz.-Nagy and
Foias(\cite{nazy}). In \cite{popescu2}, Popescu gave the
characteristic function for an infinite sequence of non-commuting
operators. The same for a commuting contractive tuple of operators
was invented by Bhattacharyya, Eschmeier and Sarkar,
\cite{tirtha-joydeb}. Popescu's characteristic function for a
non-commuting tuple, when specialized to a commuting one, gives
the same function. Given two contractions $P$ and $P_1$ on Hilbert
spaces $\mathcal H$ and $\mathcal H_1$, the characteristic
functions of $P$ and $P_1$ are said to coincide if there are
unitary operators $\sigma:\mathcal D_P \rightarrow \mathcal
D_{P_1}$ and $\sigma_*:\mathcal D_{P^*}\rightarrow \mathcal
D_{P_1^*}$ such that the following diagram commutes for all $z\in
\mathbb D$: \setlength{\unitlength}{3mm}
 \begin{center}
 \begin{picture}(20,14)(0,0)
 \put(2,3){$ \mathcal D_{P_1}$} \put(10,3){$ \mathcal D_{P_1 ^*}$}
 \put(5.6,2.2){$ \Theta_{P_1} (z)$}
 \put(1.5,6.5){$ \sigma$} \put(11,6.5){$ \sigma_*$}
 \put(2,10){$ \mathcal D_{P}$} \put(10,10){$ \mathcal D_{P ^*}$}
 \put(5.6,11){$ \Theta_{P} (z)$}
 \put(3.5,3.5){ \vector(1,0){6}} \put(3.5,10.5){ \vector(1,0){6}}
 \put(2.4,9.2){ \vector(0,-1){5}} \put(10,9.2){ \vector(0,-1){5}}

 \end{picture}
 \end{center}

The following result is due to Sz.Nagy and Foias.
\begin{thm}
Two completely non-unitary contractions are unitarily equivalent
if and only if their characteristic functions coincide.
\end{thm}
Let $(S,P)$ and $(S_1,P_1)$ be two pure $\Gamma$-contractions on
Hilbert spaces $\mathcal H$ and $\mathcal H_1$ respectively. As we
mentioned in section 1, the complete unitary invariant that we
shall produce has two contents namely the equivalence of the
fundamental operators of $(S^*,P^*)$ and $(S_1^*,P_1^*)$ and the
coincidence of the characteristic functions of $P$ and $P_1$.
\begin{prop}\label{end-prop}
If two $\Gamma$-contractions $(S,P)$ and $(S_1,P_1)$ defined on
$\mathcal H$ and $\mathcal H_1$ respectively are unitarily
equivalent then so are their fundamental operators $F$ and $F_1$.
\end{prop}
\begin{proof}
Let $U:\mathcal H \rightarrow \mathcal H_1$ be a unitary such that
$US=S_1U$ and $UP=P_1U$. Then clearly $UP^*=P_1^*U$ and
consequently
$$U{D}_P^2=U(I-P^*P)=(U-P_1^*UP)=(U-P_1^*P_1U)={D}_{P_1}^2U,$$
which implies that $U{D}_P={D}_{P_1}U$. Let $V=U|_{\mathcal
{D}_P}$. Then $V\in \mathcal B(\mathcal D_P, \mathcal D_{P_1})$
and $VD_P=D_{P_1}V$. Now
\begin{align*}
D_{P_1}VFV^*D_{P_1} & = VD_PFD_PV^* \\& = V(S-S^*P)V^* \\& =
S_1-S_1^*P_1 \\& = D_{P_1}F_1D_{P_1}.
\end{align*}
Thus $F_1=VFV^*$ and the proof is complete.
\end{proof}

The next result is a partial converse to the previous proposition
for pure $\Gamma$-contractions.

\begin{prop}\label{end-prop1}
Let $(S,P)$ and $(S_1,P_1)$ be two pure $\Gamma$-contractions on
$\mathcal{H}$ and $\mathcal H_1$ respectively such that the
characteristic functions of $P$ and $P_1$ coincide. Also suppose
that the fundamental operators $F_*$ of $(S^*,P^*)$ and ${F_1}_*$
of $(S_1^*,P_1^*)$ are unitarily equivalent by the unitary from
$\mathcal D_{P^*}$ and $\mathcal D_{P_1^*}$ that establishes the
coincidence of the characteristic functions of $P$ and $P_1$. Then
$(S,P)$ and $(S_1,P_1)$ are unitarily equivalent.
\end{prop}
\begin{proof}
Let $\mu_1\,:\,\mathcal D_P \rightarrow \mathcal D_{P_1}$ and
$\eta_1\,:\, \mathcal D_{P^*} \rightarrow \mathcal D_{P_1^*}$ be
unitaries such that the following diagram

\setlength{\unitlength}{3mm}
 \begin{center}
 \begin{picture}(20,14)(0,0)
 \put(2,3){$ \mathcal D_{P_1}$} \put(10,3){$ \mathcal D_{P_1 ^*}$}
 \put(5.6,2.2){$ \Theta_{P_1} (z)$}
 \put(1.5,6.5){$ \mu_1$} \put(11,6.5){$ \eta_1$}
 \put(2,10){$ \mathcal D_{P}$} \put(10,10){$ \mathcal D_{P ^*}$}
 \put(5.6,11){$ \Theta_{P} (z)$}
 \put(3.5,3.5){ \vector(1,0){6}} \put(3.5,10.5){ \vector(1,0){6}}
 \put(2.4,9.2){ \vector(0,-1){5}} \put(10,9.2){ \vector(0,-1){5}}

 \end{picture}
 \end{center}
 commutes for all $z\in \mathbb D$ and $\eta_1 F_*={F_1}_*
 \eta_1$. Let us define
$$\eta=(I\otimes\eta_1):H^2(\mathbb D)\otimes\mathcal{D}_{P^*}\rightarrow H^2(\mathbb D)\otimes
\mathcal{D}_{P_1^*}.$$ Since $\eta_1 \Theta_P=\Theta_{P_1}\mu_1$,
we have for any $f\in H^2(\mathbb D)\otimes\mathcal D_P$
$$ \eta (\textup{Ran} M_{\Theta_p}f)=\eta_1\Theta_P f=\Theta_{P_1}\mu_1f=M_{\Theta_{P_1}}(\mu_1 f). $$
Therefore,
$$ \eta(\mathbb{H}_P)=\mathbb{H}_{P_1},\; \textup{ as } \mathbb{H}_P=\textup{Ran}(M_{\Theta_P})^{\perp}
\textup{ and }
\mathbb{H}_{P_1}=\textup{Ran}(M_{\Theta_{P_1}})^{\perp}.
$$

Now clearly $$\eta(M_z \otimes I_{\mathcal{D}_{P^*}})^*=(M_z
\otimes I_{\mathcal{D}_{P_1^*}})^* \eta ,$$ which shows that
$\eta(\mathbb{H}_P)$ i.e, $\mathbb{H}_{P_1}$ is co-invariant under
$M_z \otimes I_{\mathcal{D}_{P_1^*}}$ and $P_{\mathbb{H}_P}(M_z
\otimes I_{\mathcal{D}_{P^*}})|_{\mathbb{H}_P}$ coincides with
$P_{\mathbb{H}_{P_1}}(M_z \otimes
I_{\mathcal{D}_{P_1^*}})|_{\mathbb{H}_{P_1}}$, i.e, $P$ defined on
$\mathcal H$ coincides with $P_1$ defined on $\mathcal H_1$.\\
Again
\begin{align*}
\eta(I \otimes {F_*}^*+ M_z\otimes F_*)^* & = \eta(I\otimes F_*+
M_z^*\otimes {F_*}^*)\\&=I\otimes \eta_1 F_*+ M_z^* \otimes \eta_1
{F}_*^*\\&= I\otimes {F_1}_* \eta_1 + M_z^* \otimes
{F_1}_*^*\eta_1
\\&= (I\otimes {F_1}_* +M_z^*\otimes {F_1}_*^*)(I\otimes
\eta_1)\\&=(I\otimes {F_1}_*^*+M_z\otimes {F_1}_*)^*(I\otimes
\eta_1),
\end{align*}
which shows that $S(\equiv P_{\mathbb{H}_P}(I\otimes
{F}_*^*+M_z\otimes F_*)|_{\mathbb{H}_P})$ and $S_1(\equiv
P_{\mathbb{H}_{P_1}}(I \otimes {F_1}_*^*+M_z\otimes
{F_1}_*)|_{\mathbb{H}_{P_1}})$ are unitarily equivalent. Hence
$(S,P)$ and $(S_1,P_1)$ are also unitarily equivalent and the
proof is complete.
\end{proof}
Combining the last two propositions we obtain the main result of
this section.
\begin{thm}
Let $(S,P)$ and $(S_1,P_1)$ be two pure $\Gamma$-contractions on
Hilbert spaces $\mathcal H$ and $\mathcal H_1$ respectively and
let $F_*$ and ${F_1}_*$ be the fundamental operators of
$(S^*,P^*)$ and $(S_1^*,P_1^*)$. Then $(S,P)$ is unitarily
equivalent to $(S_1,P_1)$ if and only if the characteristic
functions of $P$ and $P_1$ coincide and $F_*$ and ${F_1}_*$ are
unitarily equivalent by the unitary from $\mathcal D_{P^*}$ and
$\mathcal D_{P_1^*}$ that establishes the coincidence of the
characteristic functions of $P$ and $P_1$.
\end{thm}
\begin{proof}
Since $(S,P)$ and $(S_1,P_1)$ are unitarily equivalent, so are
$(S^*,P^*)$ and $(S_1^*,P_1^*)$. Now we apply Proposition
\ref{end-prop} to the $\Gamma$-contractions $(S^*,P^*)$ and
$(S_1^*,P_1^*)$ to have the unitary equivalence of $F_*$ and
${F_1}_*$.
\end{proof}

\textbf{Acknowledgements.} The first author was supported by the
Department of Science and Technology, India through the project
numbered SR/S4/MS:766/12, University Grants Commission, India via
DSA-SAP and UKIERI. and the second author was supported by
Research Fellowship of Council of Science and Industrial Research,
India and UKIERI.


\end{document}